\documentclass[a4paper,10pt]{article}
\usepackage{pgf,tikz}
\usetikzlibrary{arrows}
\usepackage{geometry}
\usepackage{latexsym}
\usepackage{amsfonts, amsmath, amssymb, amsthm}
\usepackage[utf8]{inputenc}
\usepackage[T1]{fontenc}
\usepackage[english]{babel}
\usepackage{stmaryrd}
\usepackage{subfig}
\usepackage{frcursive}
\usepackage{mathrsfs}
\usepackage{enumitem}

\newcommand{\N}{\mathbb{N}}
\newcommand{\Z}{\mathbb{Z}}

\newcommand{\R}{\mathbb{R}}
\newcommand{\C}{\mathbb{C}}

\newcommand{\E}{\mathbb{E}}

\newcommand{\hA}{\hat{A}}
\newcommand{\bin}[1]{\underline{#1}_{_2}}
\newcommand{\maineq}{s_2(n+a)-s_2(n)=d}
\newcommand{\F}[3]{\mathcal{F}_{(\alpha^{#1},\beta^{#2})}^{({#3})}}
\newcommand{\D}[2]{\mathcal{D}_{#1}(#2)}
\newcommand{\Ps}{\mathcal{P}}

\newtheorem{theorem}{Theorem}[section]
\newtheorem{lemma}[theorem]{Lemma}
\newtheorem{proposition}[theorem]{Proposition}

\theoremstyle{definition}
\newtheorem{remark}[theorem]{Remark}
\newtheorem{definition}[theorem]{Definition}
\newtheorem{example}[theorem]{Example}

\title{Central Limit Theorem for probability measures defined by sum-of-digits function in base 2}
\author{Jordan Emme and Pascal Hubert}
\date{}
\begin{document}

\maketitle

\begin{abstract}
In this paper we prove a central limit theorem for some probability measures defined as asymptotic densities of integer sets defined via sum-of-digit-function. To any non-negative integer $a$ we can associate a measure on $\mathbb{Z}$ called $\mu_a$ such that, for any $d$, $\mu_a(d)$ is the asymptotic density of the set of non-negative integers $n$ such that $s_2(n+a)-s_2(n)=d$ where $s_2(n)$ is the number of digits ``1'' in the binary expansion of $n$. We express this probability measure as a product of matrices whose coefficients are operators of $l^1(\mathbb{Z}$). Then we take a sequence of integers $(a_X(n))_{n\in\mathbb{N}}$ defined via a balanced Bernoulli sequence $X$. We prove that, for almost every sequence, and after renormalization by the typical variance, we have a central limit theorem by computing all the moments and proving that they converge towards the moments of the normal law $\mathcal{N}(0,1)$.
\end{abstract}

\section{Introduction}

\subsection{Background}

In this paper we study some properties of sets defined via sum-of-digit function in base 2. Namely, for a given integer $a$, we study the asymptotic density of set of integers such that the difference of $1$ in their binary expansion before and after addition with $a$ is a given integer. More precisely, we define:

$$
\forall n \in \N,\ s_2(n)=\sum_{k=0}^m n_k
$$
where
$$
n=\sum_{k=0}^m n_k 2^k,
$$
and
$$
\forall a \in \N,\ \forall d \in \Z,\ \mu_a(d)=\lim_{N\rightarrow +\infty}\frac{1}{N}\#\{n<N\ |\ s_2(n+a)-s_2(n)=d \}.
$$

This can be linked with correlation functions that are studied in \cite{besineau} for instance, or with properties of sum-of-digit functions which have been extensively studied in, for instance \cite{delange_somme} or, more recently, in \cite{MauduitRivat1}. We can also quote \cite{Keane} for the links between Thue-Morse sequence and the sum-of-digits function in base 2. The type of questions answered in this article also share some similarity with \cite{kopparty} and \cite{kopparty_student}.

More precisely, we are interested in normality properties of such sets and we give in this paper a central limit theorem for a random $a$. This kind of properties have raised a considerable interest in number theory. We can quote \cite{delange_qadd,drmota_joint,kim} for some of these normality properties for $q$-additive functions.

In \cite{petitlapin}, we were interested in those densities of sets and more precisely in their asymptotic properties as $a$ goes to infinity. The methods for computing those densities were essentially combinatorial. In this paper, we are closer to dynamical systems as we study a random product of matrices.

\subsection{Results}
\begin{definition}
 Let $n \in \N$. There exists a unique smallest $m \in \N$ and a unique sequence $\{{n_0},...,{n_m}\} \in \{0,1\}^m$ such that:
 $$
 n=\sum_{k=0}^{m} n_k \cdot 2^k.
 $$
 Denote $\bin{n}=n_m...{n_0}.$ the word $\bin(n)$ is an element of the free monoid $\{0,1\}^*$.
\end{definition}

\begin{definition}
 Define the sum-of-digits function in base 2 $s_2$, as:
$$ 
\begin{array}{cccc}
  s_2:&\N   &\rightarrow  &\N \\
      & n   &\mapsto      &\displaystyle{\sum_{k=0}^m n_k}
 \end{array}
$$
where $\bin{n}=n_m...n_0$.
\end{definition}

We are interested in the following equation with parameters $a \in \N$, $d\in \Z$ and unknown $n \in \N$:
$$\maineq .$$

More precisely, we wish to understand the asymptotic densities of the following sets:
$$
\mathcal{E}_{a,d}:=\left\{n\in\N\ |\ \maineq  \right\}.
$$

In \cite{petitlapin}, we prove the following:

\begin{proposition}\label{p.solution}
 For any $a \in \N$, $d\in \Z$, there exists a finite set of words $\Ps_{a,d}:=\{w_1,...,w_k\} \subset \{0,1\}^*$ such that:
 $$
 \mathcal{E}_{a,d}=\bigcup_{i \in \{1,...,k\}} [w_i]
 $$
 where $[w]$ is the set of integers $n$ such that $\bin{n}$ ends with $w$.
\end{proposition}

\begin{remark}
 Remark that  $\Ps_{a,d}$ is finite and, possibly, empty (whenever $d>s_2(a)$ actually). 
\end{remark}

From Proposition \ref{p.solution}, it is clear that the densities of the sets $\mathcal{E}_{a,d}$ exist. This was known since \cite{besineau}. In this way we define the main object of our study:

\begin{definition}
 Let us define, for any $a \in \N$, the probability measure $\mu_a$ by:
 $$
 \forall d \in \Z, \  \mu_a(d):= \lim_{N \rightarrow + \infty} \frac{\#\left\{n \leq N \ | \ \maineq \right\}}{N}.
 $$
\end{definition}

\begin{remark}
Remark that in order to get a probability measure, we cannot exchange the roles of $d$ and $a$. Indeed, one can check that for any $d$, the sequence $\left(\mu_a(d) \right)_{a \in \N}$ is not even in $l^1(\N)$.
\end{remark}

We are interested in asymtotic properties of the measures $\mu_a$ as $a$ goes to infinity in a certain sense. Namely, we define the following quantity:

\begin{definition}
 Let $a \in \N$. define the quantity:
 $$
 l(a):=\#\left\{ \text{occurences of ``01'' in }\bin{a} \right\}
 $$
\end{definition}

We recall two theorems from \cite{petitlapin}:

\begin{theorem}
 There exists a constant $C>0$ such that, for any $a \in \N$:
 $$
 \|\mu_a\|_2\leq C\cdot l(a)^{-1/4}.
 $$
\end{theorem}

\begin{theorem}
 For any $a \in \N$, the probability measure $\mu_a$ has mean 0 and its variance is bounded by:
 $$
 l(a)-1\leq \text{Var}(\mu_a)\leq 4l(a)+2.
 $$
\end{theorem}

This last theorem raises the question of whether the ratio $\frac{\text{Var}(\mu_a)}{l(a)}$ converges as $l(a)$ goes to infinity. We do this in the generic case for the balanced Bernoulli measure. More precisely, we have the following central limit theorem:

\begin{theorem}\label{t.tcl}
 Let $X=(X_n)_{n\in\N}\in\{0,1\}^\N$ be a generic sequence for the balanced Bernoulli measure. Define the sequence $\left( a_X(n) \right)_{n\in\N}$ in the following way:
 $$
 a_X(n)=\sum_{k=0}^n X_k\cdot 2^k.
 $$
 For any $n\in\N$, let $\widetilde{\mu}_{a_X(n)}\in l^1\left(\sqrt{\frac{2}{n}}\Z\right)$ defined by
 $$
 \forall d \in \sqrt{\frac{2}{n}}\Z,\  \widetilde{\mu}_{a_X(n)}=\mu_{a_X(n)}\left(\sqrt{\frac{n}{2}}d\right).
 $$
 Then 
 $$
 \widetilde\mu_{a_X(n)}\underset{n\rightarrow +\infty}{\overset{weak}{\longrightarrow}}\varphi
 $$
 where
 $$
 \begin{array}{cccc}
  \varphi :& \R &\rightarrow &\R \\
           &  t & \mapsto & \frac{1}{\sqrt{2\pi}}e^{\frac{-1}{2}t^2}
 \end{array}
 $$
\end{theorem}

\begin{remark}
 An equivalent formulation, with the same notations, is:
$$
\forall x \in \R,\ \lim_{n\rightarrow +\infty}\lim_{N\rightarrow+\infty}\frac{1}{N}\#\left\{m\leq N \ | \ \frac{s_2(m+a_X(n))-s_2(m)}{\sqrt{\frac{n}{2}}}\leq x  \right\}=\Phi(x)
$$
where $\Phi$ is the repartition function of the normal law $\mathcal{N}(0,1)$.
\end{remark}

Finally we wish to state \textsc{Cusick}'s conjecture.

Define the following quantity:
$$
\forall a \in \N,\ c_a:= \lim_{N\rightarrow+\infty}\frac{1}{N}\#\{n \leq N\ | \ s_2(n+a)\geq s_2(n) \}.
$$

The conjecture consists of two parts:
 $$
 \forall a \in \N, c_a \geq \frac12
 $$
 and
 $$
 \liminf_{a\rightarrow+\infty}c_a=\frac12.
 $$
 
 This question arose as \textsc{Cusick} was working on a similar combinatorial conjecture in \cite{cusick}. For recent advances on this, one can look at \cite{MorgenbesserSpiegelhofer} and \cite{drmota_kauers_spiegel}. In particular, the main theorem of this paper has for immediate corollary that $\frac12$ is an accumulation point for the sequence $(c_a)_{a\in\N}$.
 
 Moreover, the proof of our theorem answers a question left open in \cite{drmota_kauers_spiegel} since it states that the difference $s_2(n+a)-s(n)$ is ``usually normally distributed with mean zero and variance $\frac{|a|_2}{2}$'' where $|a|_2$ denotes the length of $a$ in base $2$.

\subsection{Outline of the paper}

The goal of this paper is to demonstrate Theorem \ref{t.tcl} by a moments method. Namely, given a sequence of probability measures, we prove the weak convergence towards the normal law $\mathcal{N}(0,1)$ by proving that all the moments of this sequence converge towards the moments of the normal law.

This article is organised as follows:

Section \ref{s.measures} deals with the measures $\mu_a$ that we already studied in \cite{petitlapin}. We recall some of their properties (and most importantly a recurrence relation between them) and write them as a finite product of matrices whose coefficients are operators on $l^1(\Z)$. This is a convenient form for our study since it allows to compute the Fourier transforms of these measures explicitly.

Section \ref{s.tcl} is devoted to the proof of Theorem \ref{t.tcl}. It is divided in the following way:

In Subsection \ref{s.characteristic}, we explicit this Fourier transform and give its Taylor expansion around 0 at order 2. This is what we need in order to compute the variance of $\mu_a$. We also mention that the characteristic function can be written as a product of matrices, which is the form that will be studied throughout the article.

Subsection \ref{s.variance} is devoted to the computation of the variance of $\mu_a$. First we do this in the general case and give an explicit formula depending only on the binary decomposition of $a$. We remark that this expression depends on some correlations of sequences in $\{0,1\}^\N$. 

Then, in Subsection \ref{s.generic}, we want to compute the ``generic'' behaviour of $\mu_a$ (meaning for a $a$ whose binary expansion is given by a balanced Bernoulli sequence). For this, we use a result in \cite {cassaigne_mauduit_sarkozy} to estimate the correlation terms. It appears that in the generic case, the variance is approximately $\frac{|a|_2}{2}$ (where $|a|_2$ is the length of $\bin{a}$. So we know that in order to get a central limit theorem, we have to renormalize $\mu_a$ by the squareroot of its variance namely $\sqrt{\frac{|a|_2}{2}}$.

In Subsection \ref{s.contribution}, since we have to compute all the moments, we need to know all the coefficients in the Taylor expansion of the characteristic function but it seems difficult to give their expression in the general case. Hence we wish to understand how ``big'' the different terms are in order to know which one will be killed by the renormalization and which one will contribute. To that end, we classify the terms in the Taylor expansion and we bound them using some algebraic properties of the matrices involved in this computation. This in turn gives bounds on the moments.

Finally, in Subsection \ref{s.moments}, we show that the moments converge towards the moments of the normal law. Thanks to the study of the contributions from the previous section, we are limited to actually computing the terms that have a chance to contribute. Some elementary linear algebra and the study of the correlations appearing in Section \ref{s.variance} are the essential tools for this.

We would like to underline the fact that the way we compute the moments limits gives a bound on the speed of convergence of the moments. However, this speed is dependant on the order of the moment and thus gives no clue as to the speed of convergence of the measures towards the normal law. This could be further studied along with a local limit theorem.

\subsection{Acknowledgement}

We wish to thank Christian Mauduit and Jo\"{e}l Rivat for their interest in this problem and for sharing their knowledge in the historical and scientific background of this question. Of course we have to thank Alexander Bufetov for his precious help regarding the moments method, especially for giving the reference needed. We also would like to thank Thomas Stoll for mentionning Cusick's Conjecture. We thank Julien Cassaigne for his useful remarks on the variance properties. Finally, we thank Lukas Spiegelhofer for making us aware of the works in \cite{drmota_kauers_spiegel}.

\section{Measures $\mu_a$ on $\Z$}\label{s.measures}

Let us start by remarking the following:
\begin{remark}
 Let $a \in \N$ and $d \in \Z$.
 $$
 \mathcal{P}_{2a,d}=\left\{ w0\ |\ w \in \mathcal{P}_{a,d} \right\}\cup\left\{ w1\ |\ w \in \mathcal{P}_{a,d} \right\}
 $$
 and
 $$
 \mathcal{P}_{2a+1,d}=\left\{ w0\ |\ w \in \mathcal{P}_{a,d-1} \right\}\cup\left\{ w1\ |\ w \in \mathcal{P}_{a+1,d+1} \right\}.
 $$
 For more details about this remark we refer the reader to  \cite{petitlapin}.
 
 From this we deduce the following:

\end{remark}

\begin{proposition}\label{p.recurrence}
 For any $a \in \N$:
 $$
 \mu_{2a}=\mu_a
 $$
 and
 $$
 \mu_{2a+1}(d)=\frac12 \mu_{a}(d-1)+\frac12 \mu_{a+1}(d+1).
 $$
\end{proposition}

\begin{remark}
 Notice that a probability measure on $\Z$ is, in particular, an element of $l^1(\Z)$. In all that follows, for simplicity of writing, we will always identify a measure on $\Z$ with its associated sequence in $l^1(\Z)$. In particular, we see $\mu_a$ both as a measure and as an element of $l^1(\Z)$. Let us define the shift $S$ on $l^1(\Z)$.
 $$
 \begin{array}{cccc}
  S: & l^1(\Z)                       & \rightarrow &l^1(\Z)\\
     &  \left(x_n\right)_{n\in \Z}   & \mapsto     &\left(x_{n+1}\right)_{n\in \Z}.
 \end{array}
 $$
 
 Then, the identities of Proposition \ref{p.recurrence} can be written:
 $$
 \mu_{2a}=\mu_a
 $$
 and
 $$
 \mu_{2a+1}=\frac12 S^{-1}(\mu_a)+\frac12 S (\mu_{a+1}).
 $$
\end{remark}

\begin{example}
It is easy to see that $\mu_0=\delta_0$.
Then, either by standard computation or using Proposition \ref{p.recurrence}, one obtains:
$$
\mu_1=\frac14 \underset{n \leq 1}{\sum_{n \in \Z}} \delta_n \cdot 2^n.
$$
\end{example}

\begin{proposition}\label{p.distrib}
 For any $a \in \N$,
$$ \mu_a = \begin{pmatrix}Id & 0 \end{pmatrix} \; A_{a_0} \cdots A_{a_{n-1}} A_{a_n}\begin{pmatrix} \mu_0 \\ \mu_1 \end{pmatrix},
$$
where the sequence $\bin{a}=a_n...a_0$ and

$$
    A_0 = \begin{pmatrix} Id &  0  \\ \frac12 {S^{-1}}& \frac12 {S} \end{pmatrix}, \qquad 
    A_1 = \begin{pmatrix} \frac12 {S^{-1}} & \frac12 {S} \\ 0 & Id \end{pmatrix}.
$$
\end{proposition}

\begin{proof}
 Suffice to remark that for any $a \in \N$,
 $$
 A_0 \begin{pmatrix} \mu_a \\ \mu_{a+1} \end{pmatrix}=\begin{pmatrix} \mu_{2a} \\ \mu_{2a+1} \end{pmatrix}
 $$
 and
 $$
 A_1 \begin{pmatrix} \mu_a \\ \mu_{a+1} \end{pmatrix}=\begin{pmatrix} \mu_{2a+1} \\ \mu_{2a+2} \end{pmatrix}
 $$
 with Proposition \ref{p.recurrence}.
\end{proof}

Notice that this proposition is a clearer version of Theorem 1.2.1 in \cite{petitlapin}.

\section{Central limit theorem}\label{s.tcl}

\subsection{Characteristic function}\label{s.characteristic}

Let $a \in \N$ with $\bin{a}=a_n...a_0$. The characteristic function of $\mu_a$, denoted $\widehat{\mu_a}$ is defined in the standard way:
$$
\forall \theta \in [0,2\pi),\ \widehat{\mu_a}(\theta)=\sum_{d \in \Z}e^{id\theta}\mu_a(d)
$$

By Proposition \ref{p.recurrence}, the characteristic function $\widehat{\mu_a}:[0,2\pi)\rightarrow \C$ is given by:
$$
\forall \theta \in [0,2\pi),\ \widehat{\mu_a}(\theta)=\begin{pmatrix}1 & 0 \end{pmatrix} \; \hA_{a_0} \cdots \hA_{a_{n-1}} \hA_{a_n}\begin{pmatrix} \widehat{\mu_0}(\theta) \\ \widehat{\mu_1}(\theta) \end{pmatrix}
$$
where
$$
  \forall \theta \in [0,2\pi), \ 
  \hA_0(\theta):= \begin{pmatrix} 1 &  0  \\\frac12 e^{i \theta}& \frac12 e^{-i \theta} \end{pmatrix}, \ 
   \hA_1(\theta):= \begin{pmatrix} \frac12 e^{i \theta} & \frac12 e^{-i \theta}\\  0 & 1 \end{pmatrix}.
$$
Indeed, from the recurrence relations in Proposition \ref{p.recurrence}, one has
$$
\forall \theta \in [0,2\pi),\ \widehat{\mu_{2a}}(\theta)=\widehat{\mu_a}(\theta)
$$
and
$$
\forall \theta \in [0,2\pi),\ \widehat{\mu_{2a+1}}(\theta)=\frac12 e^{i\theta} \widehat{\mu_a}(\theta)+\frac12 e^{-i\theta} \widehat{\mu_{a+1}}(\theta).
$$
These recurrence relations on the characteristic function justify its writting as a product of matrices.

A quick computation yields
$$
\forall \theta \in [0, 2\pi),\quad \widehat{\mu_0}(\theta)=1,\quad \widehat{\mu_1}(\theta)=\frac{e^{i\theta}}{2-e^{-i\theta}}
$$
and so,
$$
\forall \theta \in [0,2\pi),\ \widehat{\mu_a}(\theta)=\begin{pmatrix}1 & 0 \end{pmatrix} \; \hA_{a_0} \cdots \hA_{a_{n-1}} \hA_{a_n}\begin{pmatrix} 1 \\ \frac{e^{i\theta}}{2-e^{-i\theta}}\end{pmatrix}.
$$
Let us now define the matrices playing a role in the Taylor expansion of $\widehat{\mu_a}$ near 0.
$$
  I_0 = \begin{pmatrix}1 &0 \\\frac12&\frac12\end{pmatrix}, \quad\; 
  \alpha_0 = \frac{1}2 \begin{pmatrix}0&0\\ 1&-1\end{pmatrix}, \quad\; 
  \beta_0 = \frac12 \begin{pmatrix}0&0\\ 1&1\end{pmatrix}, 
$$
$$
  I_1 = \begin{pmatrix}\frac12 &\frac12\\0&1\end{pmatrix}, \quad\; 
  \alpha_1 = \frac{1}2 \begin{pmatrix}1&-1\\0&0\end{pmatrix}, \quad\; 
  \beta_1 = \frac12 \begin{pmatrix}1&1\\0&0\end{pmatrix}.
$$
Indeed, we have:

$$
\hA_j(\theta)=I_j + i\theta \alpha_j-\frac12 \theta^2 \beta_j + O(\theta^3).
$$
with $j \in \{0,1\}$.

Notice that the Taylor expansion near 0 of $\theta \mapsto \frac{e^{-i\theta}}{2-e^{i\theta}}$ is:
$$
\frac{e^{i\theta}}{2-e^{-i\theta}} =1-\theta^2+O(\theta^3).
$$

\subsection{Computation of the variance}\label{s.variance}

Define the variance of $\mu_a$:
$$
\text{Var}(\mu_a)=\sum_{d \in \Z}\mu_{a}(d)d^2.
$$

\begin{theorem}\label{thm.variance}
 For any $a \in \N$ with $\bin{a}=a_n...a_0$, denote, for any $j\in \{0,...,n\},\ b_j=(-1)^{a_j+1}$. The variance of $\mu_a$ is given by the following:
  $$
 \text{Var}(\mu_a)=\frac{n+3}{2}-\frac{1}{2^{n+1}}-\frac12 \sum_{i=1}^n\sum_{k=0}^{n-i}\frac{b_{k+i}b_k}{2^{i}}+\sum_{k=0}^{n}\frac{b_{k}+b_{n-k}}{2^{k+1}}.
 $$
\end{theorem}

\begin{proof}
 Notice that the variance is given by:
 $$
 \text{Var}(\mu_a)=\begin{pmatrix}1 & 0 \end{pmatrix} \left(\beta_{a_0}+I_{a_0}\beta_{a_1}+...+I_{a_0}\cdots I_{a_{n-1}}\beta_{a_n}\right)\begin{pmatrix}1 \\ 1 \end{pmatrix} +\begin{pmatrix}1 & 0 \end{pmatrix}I_{a_0}\cdots I_{a_{n}}\begin{pmatrix}0 \\ 2 \end{pmatrix}.
 $$
 Since $\alpha_j \begin{pmatrix}1 \\1 \end{pmatrix}=0$ and $I_j\begin{pmatrix}1 \\1 \end{pmatrix}=\begin{pmatrix}1 \\1 \end{pmatrix}$ and since the variance is given by the quadratic coefficient in the Taylor expansion of the characteristic function.
 
 We now apply a change of basis to simultanously trigonalize the matrices $I_0$ and $I_1$.
 
 Let us note
 $$
 P:=\begin{pmatrix}
     1  & 1 \\
     -1 & 1 
    \end{pmatrix}, \quad 
    P^{-1}=\frac12\begin{pmatrix}
		   1 &-1\\
		   1 & 1
		   \end{pmatrix}
 $$
 and compute
 $$\forall j \in \{0,1\}, \ 
 \widetilde{I_j}:=P I_j P^{-1}=\begin{pmatrix}
                  1 & \frac{(-1)^{j+1}}{2}\\
                  0 & \frac12
                 \end{pmatrix}, \quad
 \widetilde{\beta_j}:=P \beta_j P^{-1}=\begin{pmatrix}
                  \frac12 & 0\\
                  -\frac{(-1)^{j+1}}{2}& 0
                 \end{pmatrix}
 $$
 and
 $$
 P\begin{pmatrix}1 \\1 \end{pmatrix}=\begin{pmatrix}2 \\0 \end{pmatrix}, \quad P\begin{pmatrix}0 \\2 \end{pmatrix}=\begin{pmatrix}2 \\2 \end{pmatrix}, \quad 
 \begin{pmatrix}1 & 0 \end{pmatrix}P^{-1}=\begin{pmatrix}\frac12 & -\frac12 \end{pmatrix}.
 $$
 With this change of basis, the variance becomes:
 $$
 \text{Var}(\mu_a)=\begin{pmatrix}\frac12 & -\frac12 \end{pmatrix}\left(\widetilde\beta_{a_0}+\widetilde I_{a_0}\widetilde\beta_{a_1}+...+\widetilde I_{a_0}\cdots \widetilde I_{a_{n-1}}\widetilde\beta_{a_n}\right)\begin{pmatrix}2 \\ 0 \end{pmatrix} +\begin{pmatrix}\frac12 & -\frac12 \end{pmatrix}\widetilde I_{a_0}\cdots \widetilde I_{a_{n}}\begin{pmatrix}2 \\ 2 \end{pmatrix}.
 $$
 Notice now that for any $k \in \{0,...n\}$:
 $$
 \widetilde I_{a_0}\cdots \widetilde I_{a_{k}}=\begin{pmatrix}
                                                1 & \sum_{i=0}^k\frac{b_i}{2^{k+1-i}}\\
                                                0 & \frac{1}{2^{k+1}}
                                               \end{pmatrix}
 $$
 so, for any $k \in \{0,...n-1\}$:
  $$
 \widetilde I_{a_0}\cdots \widetilde I_{a_{k}}\widetilde\beta_{a_{k+1}}=\begin{pmatrix}
                                                \frac12-\frac{b_{k+1}}{2}\sum_{i=0}^k\frac{b_i}{2^{k+1-i}}&0\\
                                                -\frac{b_{k+1}}{2^{k+2}} & 0
                                               \end{pmatrix}
 $$
 hence we get 
 $$
 \begin{array}{ll}
   \text{Var}(\mu_a)=&\begin{pmatrix}\frac12 & -\frac12 \end{pmatrix}\left(\begin{pmatrix}\frac12 & 0 \\ -\frac{b_0}{2} & 0\end{pmatrix} + \displaystyle\sum_{k=0}^{n-1} \begin{pmatrix}
                                                \frac12-\frac{b_{k+1}}{2}\displaystyle\sum_{i=0}^k\frac{b_i}{2^{k+1-i}}&0\\
                                                -\frac{b_{k+1}}{2^{k+2}} & 0
                                               \end{pmatrix}\right)\begin{pmatrix}2 \\ 0 \end{pmatrix} \\&+ \begin{pmatrix}\frac12 & -\frac12 \end{pmatrix}\begin{pmatrix}
                                                1 & \displaystyle\sum_{i=0}^n\frac{b_i}{2^{n+1-i}}\\
                                                0 & \frac{1}{2^{n+1}}
                                               \end{pmatrix}\begin{pmatrix}2 \\ 2 \end{pmatrix}.
 \end{array}
 $$
 so
  $$
   \text{Var}(\mu_a)=\begin{pmatrix}\frac12 & -\frac12 \end{pmatrix}\left(\begin{pmatrix}1  \\ -{b_0} \end{pmatrix} + \displaystyle\sum_{k=0}^{n-1} \begin{pmatrix}
                                                1-{b_{k+1}}\displaystyle\sum_{i=0}^k\frac{b_i}{2^{k+1-i}}\\
                                                -\frac{b_{k+1}}{2^{k+2}}
                                               \end{pmatrix}\right) + 
                                                1+ \displaystyle\sum_{i=0}^n\frac{b_i}{2^{n+1-i}}      -\frac{1}{2^{n+1}}
 $$
 which yields:
  $$
 \text{Var}(\mu_a)=\frac{n+3}{2}-\frac{1}{2^{n+1}}+\frac{b_0}{2}-\frac12 \sum_{0\leq i \leq k \leq n-1}\frac{b_{k+1}b_i}{2^{k+1-i}}+\sum_{k=0}^{n-1}\frac{b_{k+1}}{2^{k+2}}+\sum_{i=0}^n \frac{b_i}{2^{n+1-i}}
 $$
 and so
 $$
 \text{Var}(\mu_a)=\frac{n+3}{2}-\frac{1}{2^{n+1}}-\frac12 \sum_{0\leq i \leq k \leq n-1}\frac{b_{k+1}b_i}{2^{k+1-i}}+\sum_{k=0}^{n}\frac{b_{k}}{2^{k+1}}+\sum_{i=0}^n \frac{b_i}{2^{n+1-i}}
 $$
 which can be written
 $$
 \text{Var}(\mu_a)=\frac{n+3}{2}-\frac{1}{2^{n+1}}-\frac12 \sum_{0\leq i \leq k \leq n-1}\frac{b_{k+1}b_i}{2^{k+1-i}}+\sum_{k=0}^{n}\frac{b_{k}+b_{n-k}}{2^{k+1}},
 $$
 or even:
  $$
 \text{Var}(\mu_a)=\frac{n+3}{2}-\frac{1}{2^{n+1}}-\frac12 \sum_{i=1}^n\sum_{k=0}^{n-i}\frac{b_{k+i}b_k}{2^{i}}+\sum_{k=0}^{n}\frac{b_{k}+b_{n-k}}{2^{k+1}}.
 $$
 
\end{proof}

\subsection{Generic case of the variance}\label{s.generic}

In all that follows, we use the following notations:
\begin{itemize}
 \item $X$ denotes a sequence in $\{0,1\}^\N$ (and we endow the set $\{0,1\}^\N$ with the balanced Bernoulli probability measure.
 \item For any sequence $X$, define $a_X(n)=\displaystyle\sum_{k=0}^n X_k \cdot 2^k$
 \item As in Theorem \ref{thm.variance}, for any $X \in \{0,1\}^\N$, define the sequence $\left( b_j \right)_{j \in \N}$ by $b_j=(-1)^{X_j+1}$. 
\end{itemize}

We wish to prove the following:
\begin{proposition}\label{p.variance_generic}
 For almost every $X \in \{0,1\}^\N$ 
 $$
 \text{Var}(\mu_{a_X(n)})\underset{n\rightarrow \infty}{\sim} \frac{n}{2}.
 $$
\end{proposition}

In order to prove this proposition, we first need a technical lemma:
 
 \begin{lemma}\label{l.correlation}
Let $X\in \{0,1\}^\N$ and define the quantity: 
 $$
 C_{2,n}=\max_{M,D}\left|\sum_{k=1}^M b_{k+d_1}b_{k+d_2}\right|,
 $$
 where the maximum is taken on all $D=(d_1,d_2)$ and $M$ such that $M+d_2\leq n$.
 
 For almost every $X \in \{0,1\}^\N$ and for every $\varepsilon$ there exists $n_{\varepsilon,X}$ such that:
 $$
 \forall n \geq n_{\varepsilon,X}, \  \left|C_{2,n} \right|< n^{\frac12+\varepsilon}
 $$
 \end{lemma}
 
 \begin{proof}
In \cite{cassaigne_mauduit_sarkozy}, the following quantity is studied:
 $$
 C_m\left((b_i)_{i\in \{0,...,n\}}\right):=\max_{M,D}\left|\sum_{k=1}^M b_{k+d_1}\times...\times b_{k+d_m}\right|,
 $$
 where the maximum is taken on all $D=(d_1,...,d_m)$ and $M$ such that $M+d_m\leq n$.
 
 So we have:
 $$
 C_{2,n}= C_2\left((b_i)_{i\in \{0,...,n\}}\right).
 $$
 
 From (2.32) and (2.33) in \cite{cassaigne_mauduit_sarkozy}, we know that, for any $l \geq1$:
 
 $$
 \E\left(C_2\left((b_i)_{i\in \{0,...,n\}}\right)^{2l}\right)\leq 5n^{4+l}(4l)^l.
 $$
 
 Let $\varepsilon >0$,
 $$
  \E\left(\frac{C_{2,n}^{2l}}{n^{(\frac12+\varepsilon)2l}}\right)\leq \frac{5n^{4+l}(4l)^l}{n^{(\frac12+\varepsilon)2l}}.
 $$
 and 
 $$
 \frac{5n^{4+l}(4l)^l}{n^{(\frac12+\varepsilon)2l}}=\frac{5(4l)^l}{n^{2\varepsilon l-4}}.
 $$
 Now, if $l$ is big enough, then $2\varepsilon l-4>2$ and thus the series
 $$
 \sum_{n=1}^{+\infty}\E\left(\frac{C_{2,n}^{2l}}{n^{(\frac12+\varepsilon)2l}}\right)
 $$
 converges.
 
 By Borel-Cantelli lemma, $\frac{C_{2,n}^{2l}}{n^{(\frac12+\varepsilon)^{2l}}}\overset{a.s.}{\underset{n\rightarrow +\infty}{\longrightarrow}}0$ and thus $\frac{C_{2,n}}{n^{(\frac12+\varepsilon)}}\overset{a.s.}{\underset{n\rightarrow +\infty}{\longrightarrow}}0$.
 Hence:
 $$
 a.e. X, \ \exists n_{\varepsilon,X}, \, \forall n\geq n_{\varepsilon,X},\  \frac{C_{2,n}}{n^{(\frac12+\varepsilon)}}<1
 $$
 and thus 
 $$
 \left|C_{2,n} \right|< n^{\frac12+\varepsilon}
 $$
 for $n$ big enough.
 
\end{proof}

Let us now prove Proposition \ref{p.variance_generic}.

\begin{proof}[Proof of Proposition \ref{p.variance_generic}]
Note that, with Theorem \ref{thm.variance}, for any $n$, 
   $$
 \text{Var}(\mu_a)=\frac{n+3}{2}-\frac{1}{2^{n+1}}-\frac12 \sum_{i=1}^n\sum_{k=0}^{n-i}\frac{b_{k+i}b_k}{2^{i}}+\sum_{k=0}^{n}\frac{b_{k}-b_{n-k}}{2^{k+1}}.
 $$
 where the $b_i$ are random variables which can take value in $\{-1,1\}$ with probability $\frac12$. The only thing to prove in order to get the result is that:
 $$
 \lim_{n\rightarrow\infty}\frac{1}{n}\left(\sum_{i=1}^n\sum_{k=0}^{n-i}\frac{b_{k+i}b_k}{2^{i}}\right)=0
 $$
 since it is obvious that:
 $$
 \lim_{n\rightarrow\infty}\frac{1}{n}\left(\frac{3}{2}-\frac{1}{2^{n+1}}+\sum_{k=0}^{n}\frac{b_{k}-b_{n-k}}{2^{k+1}}\right)=0.
 $$
 
 Let us estimate $\displaystyle\sum_{i=1}^n\sum_{k=0}^{n-i}\frac{b_{k+i}b_k}{2^{i}}$.
 
With Lemma \ref{l.correlation}, for any $\varepsilon >0$, for any $n$ big enough:
 \begin{align*}
   \left|\sum_{i=1}^n\sum_{k=0}^{n-i}\frac{b_{k+i}b_k}{2^{i}}\right|& \leq \sum_{i=1}^n\frac{\left|C_{2,n}\right|}{2^i}\\
   & \leq \sum_{i=1}^n\frac{n^{\frac12 + \varepsilon}}{2^i}\\
   & \leq n^{\frac12 + \varepsilon},
 \end{align*}
which ends the proof.

\end{proof}

\subsection{Another proof for the typical variance}

Notice that we could also do things differently in order to compute the typical variance without Lemma \ref{l.correlation}. Another way to write the variance, for any integer $a$ is the following:

 \begin{align*}
   \text{Var}(\mu_a)&=\frac{n+3}{2}-\frac{1}{2^{n+1}}-\frac12 \sum_{i=1}^n\sum_{k=0}^{n-i}\frac{b_{k+i}b_k}{2^{i}}+\sum_{k=0}^{n}\frac{b_{k}+b_{n-k}}{2^{k+1}}\\
   &=\frac{n+3}{2}-\frac{1}{2^{n+1}}-\frac12 \left(\sum_{i=1}^n\sum_{k=0}^{n-i}\frac{1}{2^{i}}-2\sum_{i=1}^n\frac{\sigma_i(a)}{2^{i}}\right)+\sum_{k=0}^{n}\frac{b_{k}+b_{n-k}}{2^{k+1}}
 \end{align*}
where 
$$
\sigma_i(a)=\#\left\{ \text{occurrences of } 0w1 \text{ in $\bin{a}$}\ | \ |w|=i-1\right\}+\#\left\{ \text{occurrences of } 1w0 \text{ in $\bin{a}$}\ | \ |w|=i-1\right\}.
$$
Hence
\begin{align*}
   \text{Var}(\mu_a)
   &=\frac{n+3}{2}-\frac{1}{2^{n+1}}-\frac12 \left(\sum_{i=1}^n\frac{n-i+1}{2^{i}}-2\sum_{i=1}^n\frac{\sigma_i(a)}{2^{i}}\right)+\sum_{k=0}^{n}\frac{b_{k}+b_{n-k}}{2^{k+1}}\\
   &=\frac{n+3}{2}-\frac{1}{2^{n+1}}-\frac12 \left((n+1)(1-\frac{1}{2^n})+\sum_{i=1}^n\frac{-i}{2^{i}}-2\sum_{i=1}^n\frac{\sigma_i(a)}{2^{i}}\right)+\sum_{k=0}^{n}\frac{b_{k}+b_{n-k}}{2^{k+1}}\\
   &=1+\frac{n}{2^{n+1}}+\frac12\sum_{i=1}^n\frac{i}{2^{i}}+\sum_{i=1}^n\frac{\sigma_i(a)}{2^{i}}+\sum_{k=0}^{n}\frac{b_{k}+b_{n-k}}{2^{k+1}}
\end{align*}

Now, notice that for any $\left( b_n \right)_{n\in \N} \in \{-1,1\}^\N$:
$$
\lim_{n\rightarrow+\infty}\frac{1}{n}\left(1+\frac{n}{2^{n+1}}+\frac12\sum_{i=1}^n\frac{i}{2^{i}}+\sum_{k=0}^{n}\frac{b_{k}+b_{n-k}}{2^{k+1}}\right)=0,
$$
and that, having $\left( X_n \right)_{n\in \N}$ a sequence of independant variables indentically distributed with the balance Bernoulli measure $\mathbb{P}$, the law of large number yields:
$$
\forall i \in \N\backslash \{0\}, \ \exists \ \mathcal{U}_i\subset \{0,1\}^\N,\ \left\{    			\begin{array}{l}
      \mathbb{P}(\mathcal{U}_i)=1\\
      \text{ and }\\
      \forall X \in \mathcal{U}_i, \lim_{n \rightarrow +\infty} \frac{\sigma_i(a_X(n))}{n}=\frac12                                                                                      \end{array}\right.
$$
Define 
$$
\mathcal{U}=\bigcap_{i \in \N\backslash\{0\}}\mathcal{U}_i.
$$
Let us now prove that for every $X \in \mathcal{U}$,
$$
\lim_{n\rightarrow +\infty}\frac{1}{n}\sum_{i=1}^n\frac{\sigma_i(a_X(n))}{2^{i}}=\frac12.
$$
It is easy to see that this limit exists (for any $i \leq n$, $\sigma_i(a_X(n))\leq n$) so let us denote it by $l$. Let us write the following equality for any $n \in \N$ and any $N < n$:
$$
\frac{1}{n}\sum_{i=1}^n\frac{\sigma_i(a_X(n))}{2^{i}}=\frac{1}{n}\sum_{i=1}^N\frac{\sigma_i(a_X(n))}{2^{i}}+\frac{1}{n}\sum_{i=N+1}^n\frac{\sigma_i(a_X(n))}{2^{i}}
$$
and so, for any $n \in \N$ and any $N < n$:

\begin{align*}
   \frac{1}{n}\sum_{i=1}^N\frac{\sigma_i(a_X(n))}{2^{i}}&\leq
 \frac{1}{n}\sum_{i=1}^n\frac{\sigma_i(a_X(n))}{2^{i}} \leq
 \frac{1}{n}\sum_{i=1}^N\frac{\sigma_i(a_X(n))}{2^{i}} + 
 \frac{1}{n}\sum_{i=N+1}^n\frac{\sigma_i(a_X(n))}{2^{i}}\\
 \frac{1}{n}\sum_{i=1}^N\frac{\sigma_i(a_X(n))}{2^{i}}&\leq
 \frac{1}{n}\sum_{i=1}^n\frac{\sigma_i(a_X(n))}{2^{i}} \leq
 \frac{1}{n}\sum_{i=1}^N\frac{\sigma_i(a_X(n))}{2^{i}} +
 \sum_{i=N+1}^n\frac{1}{2^{i}}\\
  \frac{1}{n}\sum_{i=1}^N\frac{\sigma_i(a_X(n))}{2^{i}}&\leq
 \frac{1}{n}\sum_{i=1}^n\frac{\sigma_i(a_X(n))}{2^{i}}\leq
 \frac{1}{n}\sum_{i=1}^N\frac{\sigma_i(a_X(n))}{2^{i}} +
 \frac{1}{2^{N}}
\end{align*}
and so, taking the limit as $n \rightarrow +\infty$ yields:
$$
 \frac12\sum_{i=1}^N\frac{1}{2^{i}}\leq
 l \leq
 \frac{1}{2}\sum_{i=1}^N\frac{1}{2^{i}} + 
 \frac{1}{2^N}.
 $$
 This being true for all $N$, we have that:
 $$
 \forall X \in \mathcal{U},\ \lim_{n\rightarrow +\infty} \frac{\text{Var}(\mu_{a_X(n)})}{n}=\frac12.
 $$
 
 \begin{remark}
   One can also notice,  even if it is not the goal of our paper, that by doing exactly the same proof for a non balanced Bernoulli measure $(p,1-p)$, one can prove that there exists a set $\widetilde{\mathcal{U}} \subset \{0,1\}^N$ of full measure such that:
  $$
 \forall X \in \widetilde{\mathcal{U}},\ \lim_{n\rightarrow +\infty} \frac{\text{Var}(\mu_{a_X(n)})}{n}=2p(1-p).
 $$
 \end{remark}

\subsection{Upper bounds of the moments of $\mu_{a_X(n)}$}\label{s.contribution}

We now want to have bound on the moments of order $l \in \N.$ Let us first remark that the only matrices appearing in the Taylor expansion of $\hA_i(\theta)$ are $I_i$, $\alpha_i$ and $\beta_i$. Indeed:
$$
\hA_i(\theta){=}\sum_{j=0}^{+\infty}\theta^jT_{i,j},
$$
where 
$$
T_{i,0}=I_i,\quad T_{i,2j}=\frac{(-1)^j}{(2j)!}\beta_i, \quad T_{i,2j+1}=\frac{(-1)^j i}{(2j+1)!}\alpha_i
$$
\begin{remark}\label{r.matrices}
Notice that the following relations hold:
 $$
I_0\begin{pmatrix}1 \\ 1 \end{pmatrix}=I_1\begin{pmatrix}1 \\ 1 \end{pmatrix}=\begin{pmatrix}1 \\ 1 \end{pmatrix},\quad
\alpha_0\begin{pmatrix}1 \\ 1 \end{pmatrix}=\alpha_1\begin{pmatrix}1 \\ 1 \end{pmatrix}=\begin{pmatrix}0 \\ 0 \end{pmatrix}
$$
and
$$
\alpha_0 I_0=\alpha_0 I_1=\frac12 \alpha_0, \quad \alpha_1 I_0=\alpha_1 I_1=\frac12 \alpha_1.
$$
Let us insist on the fact that these relations are crucial for our proof.
\end{remark}

Let us now introduce the following norm on $2\times 2$ matrices:
$$\|M\|=\max_{i\in\{1,2\}} \left( |M_{i,1}|+|M_{i,2}| \right)$$
which is induced by $\|\cdot\|_1$ on $\R^2$.

Notice that this defines a submultiplicative norm. Moreover,
$$
\|I_0\|=\|I_1\|=\|\alpha_0\|=\|\alpha_1\|=\|\beta_0\|=\|\beta_1\|=1.
$$

Our goal is to compute all the moments of the probability measure $\mu_{a_X(n)}$ in the generic case as $n$ goes to infinity. To that end, we arrange the terms appearing in the computation into different ``types''. A type is a couple $(\alpha^p,\beta^q)$ where $p,q$ are non negative integers. They indicate the number of matrices of $\alpha$ and $\beta$ appearing in the term. For instance, a term:
$$
M=I_{a_0}\cdots I_{a_{i_0-1}}\alpha_{a_{i_0}}I_{a_{i_0+1}}\cdots I_{a_{i_1-1}}\alpha_{a_{i_1}}I_{a_{i_1+1}}\cdots I_{a_{i_2-1}}\beta_{a_{i_2}}I_{a_{i_2+1}}\cdots I_{a_n}
$$
is of type $(\alpha^2,\beta^1)$.The order of appearance of the $\alpha$ and $\beta$ does not have any influence on the type, so that 
$$
I_{a_0}\cdots I_{a_{i_0-1}}\alpha_{a_{i_0}}I_{a_{i_0+1}}\cdots I_{a_{i_1-1}}\beta_{a_{i_1}}I_{a_{i_1+1}}\cdots I_{a_{i_2-1}}\alpha_{a_{i_2}}I_{a_{i_2+1}}\cdots I_{a_n}
$$
is also of type $(\alpha^2,\beta^1)$.

Let us denote by $\F p q n$ the set of all terms of type $(\alpha^p,\beta^q)$ in the expansion of $\widehat{\mu}_{a_X}(n)$.

This notation is introduced to ease the writing of the previous formulas as well as for handling terms with the same behaviour together. For instance, the formula for the variance becomes:
$$
Var\left( \mu_{a_X(n)} \right)=\begin{pmatrix}1&0\end{pmatrix}\sum_{M \in \F 2 0 n} M\begin{pmatrix}1\\1\end{pmatrix}+\begin{pmatrix}1&0\end{pmatrix}\sum_{M \in \F 0 1 n} M\begin{pmatrix}1\\1\end{pmatrix}+\begin{pmatrix}1&0\end{pmatrix}I_{a_0}\cdots I_{a_n} \begin{pmatrix}0\\2\end{pmatrix}.
$$

Now notice that for any $2\times2$ matrix $M$, 
$$
\left|\\\begin{pmatrix}1&0\end{pmatrix}M\begin{pmatrix}1\\1\end{pmatrix}\right| \leq \|M\|
$$
so in order to find an upper bound on terms of a given type, suffices to understand:
 $$
 \sum_{M \in \F p q n}\|M\|.
 $$

\begin{definition}
 We say that a type $\left( \alpha^p,\beta^q \right)$ contributes with weight at most $k$ if:
 $$
 \sum_{M \in \F p q n} \|M\|=O(n^k)
 $$
 
\end{definition}

\begin{lemma}\label{l.contributions}
 For any pair of nonnegative integers $(p,q)$, the type  $\left( \alpha^p,\beta^q \right)$ contributes with weight at most $q$.
\end{lemma}

\begin{proof}
 We prove this lemma by induction on $p$.
 
 First notice that $\# \F 0 q n = \dbinom{n+1}{q}$. Notice also that for any $M \in \F 0 q n$, $\|M\|\leq 1$ since $\| \cdot \|$ is submultiplicative. This implies that the type  $\left( \alpha^0,\beta^q \right)$ contributes with weight $q$.
 
 Now let us assume that the type  $\left( \alpha^p,\beta^q \right)$ contributes with weight $q$ for a given $p$ and let us prove that the type $\left( \alpha^{p+1},\beta^q \right)$ has same weight.
 Now let us partition $\F {p+1} q n$. Fix $k \leq q$ and let us estimate terms that can be written $M \alpha_{a_{i_0}} I_{a_{i_0+1}}\cdots I_{a_{i_1-1}}\beta_{a_{i_1}} I_{a_{i_1+1}}\cdots I_{a_{i_2-1}}\beta_{a_{i_2}} \cdots I_{a_{i_k-1}}\beta_{a_{i_k}}$ where $M \in \F {p}{q-k}{i_0-1}$. The sum of the norms of these terms is equal to:

\begin{align*}
 &\sum_{p+q-k\leq i_0 < ... < i_k \leq n} \sum_{M \in \F {p}{q-k}{i_0-1}} \|M \alpha_{a_{i_0}} I_{a_{i_0+1}}\cdots I_{a_{i_1-1}}\beta_{a_{i_1}} \cdots \beta_{a_{i_k}}\| \\ 
 &\leq \sum_{p+q-k\leq i_0 < ... < i_k \leq n} \sum_{M \in \F {p}{q-k}{i_0-1}} \|M \alpha_{a_{i_0}} I_{a_{i_0+1}}\cdots I_{a_{i_1-1}}\| \\ 
  &\leq n^{k-1} \sum_{p+q-k\leq i_0 <  i_1 \leq n} \sum_{M \in \F {p}{q-k}{i_0-1}} \|M \alpha_{a_{i_0}} I_{a_{i_0+1}}\cdots I_{a_{i_1-1}}\| \\
  &\leq n^{k-1} \sum_{p+q-k\leq i_0 < i_1 \leq n} \|\alpha_{a_{i_0}} I_{a_{i_0+1}}\cdots I_{a_{i_1-1}}\| \sum_{M \in \F {p}{q-k}{i_0-1}} \|M\| \\
  &\underset{\text{by induction}}{\leq} n^{k-1} \sum_{p+q-k\leq i_0 \leq  i_1 \leq n} \|\alpha_{a_{i_0}} I_{a_{i_0+1}}\cdots I_{a_{i_1-1}}\| C i_0^{q-k} \\
   &\leq C n^{q-1} \sum_{p+q-k\leq i_0 <  i_1 \leq n} \|\alpha_{a_{i_0}} I_{a_{i_0+1}}\cdots I_{a_{i_1-1}}\| 
 \end{align*}
 
 and with Remark \ref{r.matrices},
 \begin{align*}
 &\leq C n^{q-1} \sum_{p+q-k\leq i_0 <  i_1 \leq n} \|\alpha_{a_{i_0}} I_{a_{i_0+1}}\cdots I_{a_{i_1-1}}\| \\
 &\leq C n^{q-1} \sum_{p+q-k\leq i_0 < i_1 \leq n} \frac{1}{2^{i_1-i_0-1}}\\
 &\leq 2Cn^{q}.
\end{align*}

This computation being valid for any value of $k$, this yields:
$$
\sum_{M \in \F {p+1} q n} \|M\|\leq 2qCn^q,
$$
which proves the lemma.
\end{proof}

\subsection{Computing all the moments}\label{s.moments}

Let us write the expansion of $\widehat \mu_a$:
$$
\widehat \mu_a(\theta)=\sum_{k=0}^{N}\frac{i^k m_k(a)}{k!}\theta^k +o(\theta^n)
$$
where
$$
m_k(a)=\sum_{d\in\Z}\mu_a(d)d^k
$$
is the moment of order $k$ of the probability measure $\mu_a$ (indeed, recall that $\mu_a$ is centered).

Now let us renormalize $\mu_{a_X(n)}$. From Proposition \ref{p.variance_generic}, we know that we have to look at $\widetilde\mu_{a_X(n)} \in l^1\left( \sqrt{\frac{2}{n}}\Z \right)$ defined by:
$$
\forall d \in \sqrt{\frac{2}{n}}\Z,\ \widetilde\mu_{a_X(n)}(d)=\mu_{a_X(n)}\left( \sqrt{\frac{n}{2}}d \right).
$$
Now notice that the characteristic function of $\widetilde\mu_{a_X(n)}$ in $\theta$ is actually:
$$
\widehat\mu_{a_X(n)}\left(\sqrt{\frac{2}{n}}\theta\right){=}\sum_{k=0}^{N}\frac{(i\sqrt2)^k m_k(a_X(n))}{n^\frac{k}{2}k!}\theta^k +o(\theta^N).
$$
Hence, for any $n \in \N$, the moments of order $k$ of the probality measure $\widetilde\mu_{a_X(n)}$, denoted by $\widetilde m_k(a_X(n))$, are:
$$
\widetilde m_k(a_X(n))=\frac{\sqrt{2^k} m_k(a_X(n))}{n^\frac{k}{2}}
$$
and thus, we wish to understand, if it exists, for any integer $k$, 
$$
\lim_{n\rightarrow+\infty}\frac{\sqrt{2^k} m_k(a_X(n))}{n^\frac{k}{2}}.
$$

\begin{lemma}\label{l.moments_limits}
 For almost every sequence $X \in \{0,1\}^\N$, for any $k \in \N$:
 $$
  \lim_{n\rightarrow+\infty}\widetilde m_{2k}(a_X(n))=\frac{(2k)!}{2^k k!}
 $$
 and
 $$
 \lim_{n\rightarrow+\infty}\widetilde m_{2k+1}(a_X(n))=0
 $$
 which are the moments of the normal law $\mathcal{N}(0,1)$.
\end{lemma}

\begin{proof}

Remark that in this proof, we only consider terms of the following type:
$$
\begin{pmatrix}1&0\end{pmatrix}\sum_{M \in \F p q n } M \begin{pmatrix}1\\1\end{pmatrix}
$$
since taking a term of degree greater than 0 in the Taylor expansion of $\frac{e^{i\theta}}{2-e^{-i\theta}}$ will not contribute because it involves terms of smaller types on $\beta$.
 
Let us remark right away that for a moment of order $2k+1$, the type of terms which could contribute the most is $\left( \alpha^1, \beta^k \right)$. With Lemma \ref{l.contributions}, this type has weight at most $k$. Moreover, there is always a finite number of types contributing to a moment. Hence  $m_{2k+1}(a_X(n))=O(n^k)$ and thus:
$$
\lim_{n\rightarrow+\infty}\frac{\sqrt{2^{2k+1}} m_{2k+1}(a_X(n))}{n^\frac{2k+1}{2}}=0.
$$
or, equivalently,
 $$
 \lim_{n\rightarrow+\infty}\widetilde m_{2k+1}(a_X(n))=0.
 $$

Next, we consider the even moments.

For a moment $m_{2k}$, from Lemma \ref{l.contributions}, the only type potentially contributing to the limit is $\left( \alpha^0, \beta^k \right)$. More precisely, in order to get that:
 $$
  \lim_{n\rightarrow+\infty}\widetilde m_{2k}(a_X(n))=\frac{(2k)!}{2^k k!}
 $$
 one must show the following identity on limits (and prove that they exist):
$$
\lim_{n\rightarrow+\infty}\frac{(i\sqrt2)^{2k} m_{2k}(a_X(n))}{n^k (2k)!}=\lim_{n\rightarrow+\infty}\left(\frac{2}{n}\right)^k\left( \frac{-1}{2} \right)^k\begin{pmatrix}1&0\end{pmatrix}\sum_{M \in \F 0 k n } M \begin{pmatrix}1\\1\end{pmatrix},
$$
since the Taylor expansion of $\hA_j$ near 0 is:
$$
\hA_j(\theta)=I_j + i\theta \alpha_j-\frac12 \theta^2 \beta_j + O(\theta^3).
$$

This equality is equivalent to:
$$
\lim_{n\rightarrow+\infty}\frac{2^{k} m_{2k}(a_X(n))}{n^k }=\lim_{n\rightarrow+\infty}\frac{(2k)!}{n^k}\ \begin{pmatrix}1&0\end{pmatrix}\sum_{M \in \F 0 k n } M \begin{pmatrix}1\\1\end{pmatrix}.
$$
In short, we must show that:
$$
\begin{pmatrix}1&0\end{pmatrix}\sum_{M \in \F 0 k n } M \begin{pmatrix}1\\1\end{pmatrix}=\frac{n^k}{2^k k!}+o(n^k),
$$
so that
$$
\lim_{n\rightarrow+\infty}\frac{2^{k} m_{2k}(a_X(n))}{n^k}=\frac{(2k)!}{2^k k!}
$$
which is the moment of order $2k$ of the normal law $\mathcal{N}\left( 0,1 \right)$.

Let $\mathcal{D}_k(n)=\{(d_1,...,d_k)\ | \ 0\leq d_1<...<d_k\leq n\}$. 
For $d \in \D k n$, denote $\Pi_{d}=\widetilde I_{a_0}\cdots\widetilde I_{a_{d_1-1}}\widetilde \beta_{a_{d_1}}\cdots\widetilde I_{a_{d_k-1}}\widetilde \beta_{a_{d_k}}$ (this is just a matrix $M \in \F 0 k n $ after the change of basis described in the proof of Theorem \ref{thm.variance}).
Let us prove by induction on $k$ that:

$$
\Pi_{d}=\begin{pmatrix}
       \frac{1}{2^k}+A_{d} & 0 \\
       B_{d} & 0
      \end{pmatrix}
$$
with $A_d$ and $B_d$ satisfying:
$$
\lim_{n\rightarrow +\infty}\frac{1}{n^k} \sum_{d \in \D k n} \left|A_{d}\right| = 0.
$$
and
$$
\lim_{n\rightarrow +\infty}\frac{1}{n^k} \sum_{d \in \D k n} \left|B_{d}\right| = 0.
$$

The case of $k=1$ is treated in Lemma \ref{l.correlation}. Let us assume this is true up to an integer $k$.

Let $d \in \D {k+1} {n}$. For clarity in the fomulas, let us write $d=(d_1...,d_{k-1},j,l)$ and $d'=(d_1,...,d_{k-1},j) \in \D k n.$
$$
\Pi_{d}=\Pi_{d'}\widetilde I_{a_{j+1}}\cdots\widetilde I_{a_{l-1}}\widetilde \beta_{a_{l}}
$$
Now compute:
$$
\widetilde I_{a_{j+1}}\cdots\widetilde I_{a_{l-1}}\widetilde \beta_{a_{l}}=\begin{pmatrix}
		    \frac12 - b_{l}\displaystyle \sum_{i=j+1}^{l-1} \frac{b_{i}}{2^{l-i}}&0\\
		    -\frac{b_{l}}{2^{l-j}}&0
		    \end{pmatrix}
$$
and by induction hypothesis, 
$$
\Pi_{d'}=\begin{pmatrix}
       \frac{1}{2^k}+A_{d'} & 0 \\
       B_{d'} & 0
      \end{pmatrix}
$$
Thus
$$
\Pi_{d}=\begin{pmatrix}
		    \frac{1}{2^{k+1}} -
		    \frac{1}{2^{k}}b_{l}\displaystyle \sum_{i=j+1}^{l-1} \frac{b_{i}}{2^{l-i}} +
		    \frac{A_{d'}}{2}-
		    A_{d' }b_{l}\displaystyle \sum_{i=j+1}^{l-1} \frac{b_{i}}{2^{l-i}}&0\\
		    \frac12 B_{d'}-
		    B_{d'}b_{l}\displaystyle \sum_{i=j+1}^{l-1} \frac{b_{i}}{2^{l-i}}&0
		    \end{pmatrix}
$$
and we have to prove the following:

\textit{Claim:}
$$
\frac{1}{n^{k+1}}\sum_{d \in \D {k+1} n}\left|-
		    \frac{1}{2^{k}}b_{l}\displaystyle \sum_{i=j+1}^{l-1} \frac{b_{i}}{2^{l-i}} +
		    \frac{A_{d'}}{2}-
		    A_{d' }b_{l}\displaystyle \sum_{i=j+1}^{l-1} \frac{b_{i}}{2^{l-i}}\right|\underset{n\rightarrow \infty}{\longrightarrow}0.
$$

\textit{Proof of the claim:}

\begin{itemize}
\item First,we have:
\begin{align*}
 \sum_{d\in \D {k+1} n}\left| \frac{1}{2^{k}}b_{l}\displaystyle \sum_{i=j+1}^{l-1} \frac{b_{i}}{2^{l-i}}\right|
 &= \sum_{j=k}^{n-k} \sum_{0\leq d_1<...< d_{k-1}<j}\sum_{l=j+1}^{n}\left| \frac{1}{2^{k}}b_{l}\displaystyle \sum_{i=j+1}^{l-1} \frac{b_{i}}{2^{l-i}}\right|\\
 &\leq n^{k-1}\sum_{j=k}^{n-k} \sum_{l=j+1}^{n}\left| \frac{1}{2^{k}}b_{l}\displaystyle \sum_{i=j+1}^{l-1} \frac{b_{i}}{2^{l-i}}\right|
 \end{align*}
 and notice that $\displaystyle\sum_{l=j+1}^{n}\left| b_{l}\displaystyle \sum_{i=j+1}^{l-1} \frac{b_{i}}{2^{l-i}}\right|\leq C_{2,n}$ so, for $n$ big enough, according to Lemma \ref{l.correlation}:
 \begin{align*}
 &n^{k-1}\sum_{j=k}^{n-k} \sum_{l=j+1}^{n}\left| \frac{1}{2^{k}}b_{l}\displaystyle \sum_{i=j+1}^{l-1} \frac{b_{i}}{2^{l-i}}\right|\\
 &{\leq} \frac{n^{k-1}}{2^k}\sum_{j=k}^{n-k} n^{\frac12 + \varepsilon}\\
 .&\leq \frac{n^{k+\frac12 + \varepsilon}}{2^{k}} 
\end{align*}
 \item Also,
 $$
\sum_{d\in \D {k+1} n}\ \left|\frac{A_{d'}}{2}\right|=\sum_{d'\in \D k n}\sum_{l=k}^{n} \left|\frac{A_{d'}}{2}\right|\leq 
n\sum_{d'\in \D k n}\left|\frac{A_{d'}}{2}\right|
$$
and, by induction,
$$
\lim_{n\rightarrow +\infty}\frac{1}{n^k}\sum_{d'\in \D {k} n}\left|\frac{A_{d'}}{2}\right|=0
$$
hence
$$
\lim_{n\rightarrow +\infty}\frac{1}{n^{k+1}}\sum_{d\in \D {k+1} n}\ \left|\frac{A_{d'}}{2}\right|=0.
$$

\item Finally:
\begin{align*}
 \sum_{d\in \D {k+1} n}\left|A_{d' }b_{l}\displaystyle \sum_{i=j+1}^{l-1} \frac{b_{i}}{2^{l-i}}\right|
 &= \sum_{j=k}^{n-k} \sum_{0\leq d_1<...< d_{k-1}<j}\sum_{l=j+1}^{n}\left| A_{d'}b_{l}\displaystyle \sum_{i=j+1}^{l-1} \frac{b_{i}}{2^{l-i}}\right|\\
 &\leq \sum_{j=k}^{n-k} \sum_{0\leq d_1<...< d_{k-1}<j}\left|A_{d'}\right|\sum_{l=j+1}^{n}\left| b_{l}\displaystyle \sum_{i=j+1}^{l-1} \frac{b_{i}}{2^{l-i}}\right|
 \end{align*}
 and, as in the study of the first term, using Lemma \ref{l.correlation}, for $n$ big enough:
 \begin{align*} 
  & \sum_{j=k}^{n-k} \sum_{0\leq d_1<...< d_{k-1}<j}\left|A_{d'}\right|\sum_{l=j+1}^{n}\left| b_{l}\displaystyle \sum_{i=j+1}^{l-1} \frac{b_{i}}{2^{l-i}}\right|\\
 &\leq \sum_{j=k}^{n-k} \sum_{0\leq d_1<...< d_{k-1}<j}\left|A_{d'}\right|n^{\frac12+\varepsilon}\\
  &\leq n^{\frac12+\varepsilon} \sum_{d'\in \D k n}\left|A_{d'}\right|\\
 &\leq n^{k+\frac12 + \varepsilon} 
\end{align*}
by induction hypothesis.
\end{itemize}

In the end,
$$
\lim_{n\rightarrow +\infty}\frac{1}{n^{k+1}}\sum_{d\in \D {k+1} n} \left| \frac{1}{2^{k}}b_{l}\displaystyle \sum_{i=j+1}^{l-1} \frac{b_{i}}{2^{l-i}} +
		    \frac{A_{d'}}{2}-
		    A_{d' }b_{l}\displaystyle \sum_{i=j+1}^{l-1} \frac{b_{i}}{2^{l-i}}\right|=0.
$$
And the same goes for proving that:
$$
\frac{1}{n^{k+1}}\sum_{d \in \D {k+1} n}\left|\frac12 B_{d'}-B_{d'}b_{l}\displaystyle \sum_{i=j+1}^{l-1} \frac{b_{i}}{2^{l-i}}\right|\underset{n\rightarrow \infty}{\longrightarrow}0.
$$

Hence,
\begin{align*}
 \begin{pmatrix}1&0\end{pmatrix}\sum_{M \in \F 0 k n } M \begin{pmatrix}1\\1\end{pmatrix}
&= \begin{pmatrix}\frac12&-\frac12\end{pmatrix}\sum_{d \in \D k n } \Pi_d \begin{pmatrix}2\\0\end{pmatrix}\\
&=\sum_{d \in \D k n} \frac{1}{2^k} + A_d-\sum_{d \in \D k n} B_d\\
&=\sum_{d \in \D k n} \frac{1}{2^k} + \sum_{d \in \D k n}A_d-\sum_{d \in \D k n} B_d\\
&=\dbinom{n}{k}\frac{1}{2^k} + \sum_{d \in \D k n}A_d-\sum_{d \in \D k n} B_d\\
&=\frac{n!}{k!(n-k)!2^k}+ \sum_{d \in \D k n}A_d-\sum_{d \in \D k n} B_d
\end{align*}
and since
$$
\lim_{n\rightarrow +\infty}\frac{1}{n^k} \sum_{d \in \D k n} \left|A_{d}\right| = 0
$$
and
$$
\lim_{n\rightarrow +\infty}\frac{1}{n^k} \sum_{d \in \D k n} \left|B_{d}\right| = 0,
$$
we have that:
$$
\begin{pmatrix}1&0\end{pmatrix}\sum_{M \in \F 0 k n } M \begin{pmatrix}1\\1\end{pmatrix}\underset{n\rightarrow+\infty}{=}\frac{n^k}{2^k k!}+o(n^k),
$$
which yields:
$$
\lim_{n\rightarrow+\infty}\frac{2^{k} m_{2k}(a_X(n))}{n^k}=\frac{(2k)!}{2^k k!}.
$$
\end{proof}

Hence the moments of the probability measure $\widetilde\mu_{a_X(n)}$ converge towards the moments of the normal law $\mathcal{N}(0,1)$, which prove the Central Limit Theorem \ref{t.tcl} by \cite{markov}.


\end{document}